\documentclass[12pt]{amsart}
\usepackage{amsmath,amssymb,amsfonts,latexsym,amscd,psfrag,mathabx,graphicx,mathrsfs,stmaryrd}
\usepackage[square,sort,comma,numbers]{natbib}

\usepackage[utf8]{inputenc}
\usepackage[T1]{fontenc}
\usepackage{lmodern}

\usepackage{tikz}
\usepackage{verbatim}
\usetikzlibrary{shapes,shadows,calc}
\usepgflibrary{arrows}
\usetikzlibrary{arrows, decorations.markings, calc, fadings, decorations.pathreplacing, patterns, decorations.pathmorphing, positioning}
\tikzset{nodc/.style={circle,draw=blue!50,fill=pink!80,inner sep=4.2pt}}
\tikzset{nod1/.style={circle,draw=black,fill=black,inner sep=1pt}}
\tikzset{nod2/.style={circle,draw=black,fill=black,inner sep=1.6pt}}
\tikzset{nod3/.style={circle,draw=black,inner sep=2pt}}
\tikzset{nodempty/.style={circle,draw=black,inner sep=2pt}}
\tikzset{nodde/.style={circle,draw=blue!50,fill=pink!80,inner sep=4.2pt}}
\tikzset{noddee/.style={circle,draw=black,fill=black,inner sep=2pt}}

\usepackage{scalefnt}
\usetikzlibrary{arrows,decorations.pathmorphing,backgrounds,positioning,fit,petri}

\usepackage{subfigure}
\usepackage{float}
\usepackage{xcolor}

\theoremstyle{plain}
\newtheorem{proposition}[equation]{Proposition}
\newtheorem{theorem}[equation]{Theorem}

\newtheorem{corollary}[equation]{Corollary}
\newtheorem{lemma}[equation]{Lemma}

\theoremstyle{definition}

\newtheorem{definition}[equation]{Definition}
\newtheorem{remark}[equation]{Remark}

\newcommand{\CE}{\mathcal{C}}

\newcommand{\A}{\mathcal{A}}

\newcommand{\kk}{\Bbbk}

\newcommand{\cc}{\mathfrak{c}}

\newcommand{\Ind}{\operatorname{Ind}}

\newcommand{\lk}{\operatorname{lk}}
\newcommand{\del}{\operatorname{del}}

\newcommand{\cd}{\operatorname{co-chord}}
\newcommand{\reg}{\operatorname{reg}}

\newcommand{\im}{\operatorname{im}}

\newcommand{\ve}{\operatorname{ve}}

\newcommand{\col}{\operatorname{col}}

\begin{document}

\bibliographystyle{plain}

\title{The $v$-number and Castelnuovo-Mumford regularity of graphs}

\author{Yusuf Civan}

\address{Department of Mathematics, Suleyman Demirel University,
Isparta, 32260, Turkey.}
\email{yusufcivan@sdu.edu.tr}


\date{\today}


\subjclass[2010]{}

\begin{abstract}
We prove that for every integer $k\geq 1$, there exists a connected graph $H_k$ such that $v(H_k)=\reg(H_k)+k$, where $v(G)$ and $\reg(G)$ denote the $v$-number and the (Castelnuovo-Mumford) regularity of a graph $G$ respectively.
\end{abstract} 

\maketitle
\section{Introduction}
A \emph{clutter} $\CE$ on a vertex set $V=\{x_1,\ldots,x_n\}$ is a family of subsets (edges or circuits) of $V$ which are pairwise incomparable with respect to the inclusion.  We identify the set of edges with the clutter $\CE$ itself. When $R=\kk[x_1,\ldots,x_n]=\bigoplus_{t=0}^{\infty} R_t$ is a polynomial ring over a ﬁeld $\kk$ with the standard grading, the \emph{edge ideal} $I(\CE)$ of the clutter $\CE$ is defined to be the ideal of $R$ generated by all squarefree monomials $x_e:=\prod_{x_i\in e} x_i$
such that $e\in \CE$.

Consider the minimal free graded resolution of $R/I(\CE)$ as an $R$-module:
\begin{equation*}
0\to \bigoplus_j R(-j)^{\beta_{t,j}}\to \ldots \to \bigoplus_j R(-j)^{\beta_{1,j}}\to R\to R/I(\CE)\to 0.
\end{equation*}
The \emph{Castelnuovo-Mumford regularity} or simply the \emph{regularity} of $R/I(\CE)$ is defined as
\begin{equation*}
\reg_{\kk}(\CE):=\reg_{\kk}(R/I(\CE))=\max \{j-i\colon \beta_{i,j}\neq 0\}\footnote{Unless otherwise stated, our results are independent of the characteristic of the coefficient field. So, wherever it is appropriate, we drop $\kk$ from our notation.}.
\end{equation*}

Most of the recent work in the area has been focused on either finding applicable bounds on the (Castelnuovo-Mumford) regularity $\reg(\CE)$ in terms of other parameters or performing exact computation of the regularity in specific cases~\cite{BC-prime}. Recently, a new invariant associated to the edge ideal $I(\CE)$ is introduced in~\cite{CSTVV} to study the asymptotic behavior of the minimum distance of projective Reed–Muller-type codes that we recall next. If $\CE$ is a clutter, a subset $A\subseteq V$ is an \emph{independent set} if $e\nsubset A$ for every $e\in \CE$. Furthermore, a subset $W\subseteq V$ is said to be a \emph{vertex cover} provided that $V\setminus W$ is an independent set.

When $A$ is an independent set, a vertex $w\in V\setminus A$ is a \emph{neighbor} of $A$ if $A\cup \{w\}$ contains an edge of $\CE$. We denote by $N_{\CE}(A)$, the set of all neighbors of $A$ in $\CE$, and by $\A(\CE)$, the family of all independent sets $A$ such that $N_{\CE}(A)$ is a minimal vertex cover.

\begin{definition}\textnormal{~\cite{CSTVV,JV}}
The $v$-\emph{number} of a clutter $\CE$ is defined by 
$$v(\CE):=\min \{|A|\colon A\in \A(\CE)\}.$$
\end{definition}
Notice the change of notation. The $v$-number is originally defined for the ideal $I(\CE)$ in ~\cite{CSTVV,JV}, while there is no harm for associating it directly to the clutter $\CE$ itself.

The subject of two recent papers~\cite{JV,SS} is the comparison of the $v$-number and the regularity of graphs. Under suitable restrictions, it is proved that the $v$-number of a graph $G$ provides a lower bound to $\reg(G)$. On the other hand, Jaramillo and Villarreal~\cite{JV} show that there exists a graph $G$ satisfying $v(G)=\reg(G)+1$ (over the field of rationals), and ask whether or not the inequality $v(H)\leq \reg(H)+1$ holds for every graph $H$.

We prove that there exist even connected graphs such that their $v$-numbers are far larger than their regularities. 

\begin{theorem}\label{thm:main}
For every integer $k\geq 1$, there exists a connected graph $H_k$ such that 
$v(H_k)=\reg(H_k)+k$.
\end{theorem}

\section{Preliminaries}
A \emph{simplicial complex} $X$ on a vertex set $V$ is simply a family of subsets of $V$, closed under inclusion such that $\{x\}\in X$ for every $x\in V$. A set $A\in X$ is said to be a \emph{face} (or a \emph{simplex}) of $X$, and the dimension of a face $A$ is $\dim(A)=|A|-1$. The \emph{dimension} $\dim(X)$ of a complex $X$ is the maximum dimension of a face in $X$. A face $F$ of $X$ is said to be a \emph{facet} of $X$ if it is a maximal face with respect to the inclusion.

For a given face $A\in X$, the \emph{deletion} and \emph{link} subcomplexes of $X$ at the face $A$ is defined by $\del(X;A):=\{S\in X\colon A\cap S=\emptyset\}$ and $\lk(X;A):=\{T\in X\colon A\cap T=\emptyset\;\textnormal{and}\;T\cup A\in X\}$. 

When $X$ is a simplicial complex, a subset $S\subseteq V$ is said to be a \emph{circuit} (\emph{minimal non-face}) of $X$ if $S$ is not a face of $X$ while any proper subset of $S$ is. We denote by $\cc(X)$, the family of all circuits of $X$. On the other hand, the family of all independent sets of a clutter $\CE$ forms a simplicial complex, the \emph{independence complex} $\Ind(\CE)$ of $\CE$. In fact, there is a one-to-one correspondence between simplicial complexes and clutters on a fixed set of vertices. If $X$ is a simplicial complex on $V$, then the set $\cc(X)$ of its circuits is a clutter, and $X=\Ind(\cc(X))$. On the other side, we have that $\cc(\Ind(\CE))=\CE$ for a clutter $\CE$ on $V$.

When $G=(V,E)$ is a (finite and simple) graph, we denote by $N_G(x):=\{y\in V\colon xy\in E\}$, the (open) neighborhood of $x$ in $G$, whereas $N_G[x]:=N_G(x)\cup \{x\}$ is its closed neighborhood. In particular,
we set $N_G(S):=\bigcup_{s\in S} N_G(s)$ for $S\subseteq V$. The size of the set $N_G(v)$ is called the \emph{degree} of $x$ in $G$ and denoted by $\deg_G(v)$. Furthermore, $\overline{G}$ denotes the complement of the graph $G$. For a given subset $S\subseteq V$, the subgraph $G[S]$ of $G$ induced by the set $S$ is the graph on $S$ with $E(G[S])=E\cap (S\times S)$. Finally, we denote by $K_n$, $P_n$ and $C_k$, the complete, path and cycle graphs on $n\geq 1$ and $k\geq 3$ vertices respectively.

We say that $G$ is $H$-free if no induced subgraph of $G$ is isomorphic to $H$. A graph $G$ is called \emph{chordal} if it is $C_r$-free for every $r>3$. Moreover, a graph $G$ is said to be \emph{co-chordal} if its complement $\overline{G}$ is a chordal graph.

The following provides an inductive bound on the regularity of graphs. 

\begin{lemma}\textnormal{\cite{DHS}}\label{lem:induction-sc}
Let $G$ be a graph and let $v\in V$ be given. Then
\begin{equation*}
\reg (G)\leq \max\{\reg(G-v), \reg(G-N_G[v])+1\}.
\end{equation*}
Moreover, $\reg(G)$ always equals to one of $\reg(G-v)$ or $\reg(G-N_G[v])+1$. 
\end{lemma}

A matching in a graph is a subset of edges no two of which share a vertex. An induced matching is a matching $M$ if no two vertices belonging to different edges of $M$ are adjacent. The maximum size of an induced matching of $G$ is known as the \emph{induced matching number} $\im(G)$ of $G$. The induced matching number provides a lower bound to regularity, that is, $\im(G)\leq \reg(G)$ holds for every graph $G$~\citep{MK}. Finally, if we denote by
$\cd(G)$, the least number of co-chordal subgraphs $G_1,\ldots,G_k$ of $G$ satisfying $E(G)=\bigcup_{i=1}^k E(G_i)$, then the inequality $\reg(G)\leq \cd(G)$ holds for every graph $G$~\citep{RW2}.
\section{The $v$-number of simplicial complexes and graphs}
Using the above stated correspondence between simplicial complexes and clutters, we set $v(X):=v(\cc(X))$ for every simplicial complex $X$. We prove that the $v$-number of simplicial complexes is closely related to a known parameter appearing in the collapsibility theory of Wegner~\cite{Weg}. In particular, we verify that the $v$-number of the independence complex of a graph corresponds a domination parameter on the underlying graph.

We next rephrase the $v$-number in the language of simplicial complexes as follows. Let $A\in X$ be a face, and define $U_X(A):=\{v\in V-A\colon A\in \lk(X;v)\}\cup A$. Observe that if $F\in X$ is a facet, then $U_X(F)=F$.

\begin{proposition}\label{prop:v-n-sc}
$v(X)=\min \{|A|\colon A\in X\;\textnormal{and}\;U_X(A)\;\textnormal{is a facet of}\;X\}$ for every simplicial complex $X$.
\end{proposition} 
\begin{proof}
Since $X=\Ind(\cc(X))$, a vertex $v\in V-A$ is a neighbor of a face $A\in X$ in $\cc(X)$ if and only if $A\notin \lk(X;v)$. Moreover, $N_{\cc(X)}(A)$ is a minimal vertex cover in $\cc(X)$ if and only if $V-N_{\cc(X)}(A)$ is a facet of $X$. However, the set $V-N_{\cc(X)}(A)$ is clearly equal to $U_X(A)$.
\end{proof}

A face $A$ of a simplicial complex $X$ is called a \emph{free face} if there exists a unique facet containing it. We define 
$$\lambda(X):=\min\{|A|\colon A\;\textrm{is a free face of}\;X\}\footnote{This is denoted by $\gamma_0(X)$ in~\citep[Section 4]{MT}.}.$$

\begin{corollary}\label{cor:v-free}
$v(X)=\lambda(X)$ for every simplicial complex $X$.
\end{corollary}
\begin{proof}
Following Proposition~\ref{prop:v-n-sc}, if $v(X)=|A|$, then $U_X(A)$ is the unique facet of $X$ containing $A$; hence, $A$ is a free face of $X$. On the other hand, if $B$ is a free face of $X$ and $F$ is the unique facet containing it, then $U_X(B)=F$.
\end{proof}

For a given simplicial complex $X$, we define $\beta(X)$ to be the maximum size of a face $A$ such that it is minimal with the property that $\lk(X;A)$ is a simplex.\footnote{The number $\beta(X)$ for a simplicial complex $X$ is firstly considered in~\citep[Theorem $5.4$]{HW}.}

\begin{proposition}
$v(X)\leq \beta(X)$ for every simplicial complex $X$.
\end{proposition}
\begin{proof}
Assume that $\beta(X)=|A|$ for some face $A\in X$. Since $\lk(X;A)$ is a simplex, it follows that $\{v\in V-A\colon A\in \lk(X;v)\}\in \lk(X;A)$, that is, $U_X(A)\in X$. If $U_X(A)$ is not a facet, then it is contained in a facet $F$ of $X$. Let $u\in F\setminus U_X(A)$ be a vertex. It then follows that $u\notin A$ and $A\in \lk(X;u)$ so that $u\in U_X(A)$, a contradiction.
\end{proof}

We say that a free face is \emph{minimal} if none of its proper subsets is a free face in $X$.
\begin{corollary}\label{cor:v-beta}
$\beta(X)$ equals to the maximum size of a minimal free face in $X$.
\end{corollary}

\begin{remark}\label{rem:russ}
The bound $v(X)\leq \beta(X)$ could be strict even for the independence complexes of graphs. For instance, we have that $v(P_4)=1<2=\beta(P_4)$. Moreover, we note that the beta-number and the regularity of simplicial complexes are incomparable in general. For the graph $P_4$, we have $\reg(P_4)=1<2=\beta(P_4)$. On the other hand, denote by $G$, the graph\footnote{The construction of the graph $G$ is due to R.~Woodroofe, and it was devised over a discussion with the author.} depicted in Figure~\ref{fig:russ} and set $H:=\overline{G}$. It can be easily checked that $v(H)=\beta(H)=1<2=\reg(H)$.
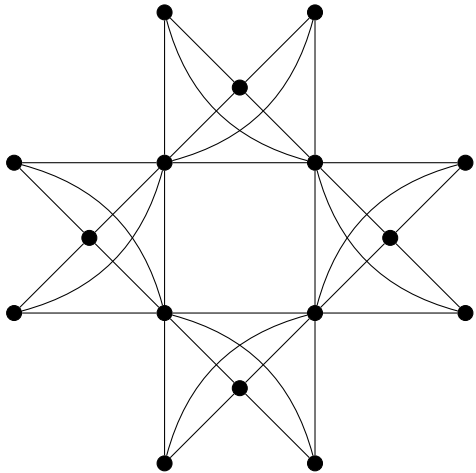
\begin{figure}[ht]
\begin{center}
\begin{tikzpicture}[scale=1]

\node [noddee] at (0,2) (v1) [] {};
\node [noddee] at (1,3) (v2) [] {}
	edge [] (v1);
\node [noddee] at (0,4) (v3) [] {}
     edge [] (v2);

\node [noddee] at (2,2) (v4) [] {}     
        edge [] (v1)
        edge [] (v2)
        edge [bend right] (v3);
\node [noddee] at (2,4) (v5) [] {}        
      edge [] (v4)
      edge [] (v3)
      edge [] (v2)
      edge [bend left] (v1);
\node [noddee] at (2,0) (v6) [] {}         
      edge [] (v4);
      
\node [noddee] at (3,1) (v7) [] {}         
      edge [] (v6)
      edge [] (v4);      
\node [noddee] at (4,0) (v8) [] {}         
      edge [] (v7)
      edge [bend right] (v4);
\node [noddee] at (4,2) (v9) [] {}         
      edge [] (v4)
      edge [] (v7)
      edge [bend right] (v6)
      edge [] (v8);
\node [noddee] at (2,6) (v10) [] {}         
      edge [] (v5);
\node [noddee] at (3,5) (v11) [] {}         
      edge [] (v5)
      edge [] (v10);      
\node [noddee] at (4,6) (v12) [] {}         
      edge [bend left] (v5)
      edge [] (v11);
\node [noddee] at (4,4) (v13) [] {}         
      edge [] (v5)
      edge [] (v9)
      edge [] (v11)
      edge [bend left] (v10)
      edge [] (v12); 
      
\node [noddee] at (5,3) (v14) [] {}         
      edge [] (v9)
      edge [] (v13);      
\node [noddee] at (6,4) (v15) [] {}         
      edge [] (v13)
      edge [] (v14)
      edge [bend right] (v9);           
\node [noddee] at (6,2) (v16) [] {}         
      edge [] (v9)
      edge [] (v14)
      edge [bend left] (v13);
\end{tikzpicture}
\end{center}
\caption{The graph $G$ in Remark~\ref{rem:russ}.}
\label{fig:russ}
\end{figure}
\end{remark}
\subsection{The $v$-number of graphs}

Let $G=(V,E)$ be a graph. A vertex $u\in V$ is said to \emph{vertex-wise dominate} an edge $e=xy\in E$, if $u\in N_G[x]\cup N_G[y]$. A subset $S\subseteq V$ is called a \emph{vertex-wise dominating set} (shortly $\ve$-dominating set), if for any edge $e\in E$, there exists a vertex $s\in S$ that vertex-wise dominates $e$ (see~\citep{JL} for details). Furthermore, a $\ve$-dominating set $S$ of $G$ is called \emph{minimal}, if no proper subset of $S$ is $\ve$-dominating for $G$. When $S$ is a minimal $\ve$-dominating set for $G$, every vertex in $S$ has a private neighbor in $E$. In other words, $e$ is a \emph{vertex-wise private neighbor} of $s\in S$ if $s$ $\ve$-dominates $e$ while no vertex in $S\setminus \{s\}$ $\ve$-dominates the edge $e$ in $G$.

The \emph{independent vertex-wise domination number} and the \emph{upper independent vertex-wise domination number} of $G$ are defined by
\begin{align*}
&i_{\ve}(G):=\min \{|S|\colon S\;\textrm{is\;an\;independent\;vertex-wise\;dominating\;set\;of\;}G\},\\
&\beta_{\ve}(G):=\max\{|T|\colon T\;\textrm{is\;a\;minimal\;independent\;vertex-wise\;dominating\;set\;of\;}G\}
\end{align*}
respectively. Notice that the inequality $i_{\ve}(G)\leq \beta_{\ve}(G)$ holds for every graph $G$.

\begin{theorem}\label{thm:v-ive}
$v(G)=i_{\ve}(G)$ for every graph $G$.
\end{theorem}
\begin{proof}
Suppose that $i_{\ve}(G)=|S|$, where $S$ is an independent $\ve$-dominating set. Observe that $N_G(S)$ is a vertex cover. Indeed, assume otherwise that there exists an edge $e=xy$ in $G-N_G(S)$. Since $S$ is a $\ve$-dominating set, there exists $s\in S$ such that either $sx\in E$ or $sy\in E$. However, this implies that either $x\in N_G(S)$ or else $y\in N_G(S)$, a contradiction.
The fact that $N_G(S)$ is a minimal vertex cover follows, since $S$ is a independent $\ve$-dominating set in $G$. This shows that $v(G)\leq i_{\ve}(G)$.

Next, assume that $v(G)=|A|$ for some independent set $A$ in $G$. If $f=uv\in E$, we must have that either $u\in N_G(A)$ or else $v\in N_G(A)$, since $N_G(A)$ is a vertex cover. However, this means that $A$ is an independent $\ve$-dominating set in $G$. Therefore, we conclude that $i_{\ve}(G)\leq |A|=v(G)$. 
\end{proof}

\begin{corollary}
$\beta_{\ve}(G)=\beta(\Ind(G))$ for every graph $G$.
\end{corollary}
\section{Proof of Theorem~\ref{thm:main}}
We consider a $17$-vertex flag triangulation of the dunce hat~\citep{BFJ} illustrated as in Figure~\ref{fig:fig1}, where vertices with the same label are identified. Denote by $D$, the graph whose independence complex is isomorphic to given triangulation.

\usetikzlibrary{positioning,backgrounds}
\begin{figure}[htb]
\centering
\begin{tikzpicture}[scale=1.24]
\tikzstyle{circ} = [circle, minimum width=1.5mm, inner sep=0pt,draw,fill]

\node[circ] (a) at (0, 0) [label=below left:{ $1$}] {};
\node[circ] (b) at (3, 0)  [label=below:{ $2$}] {};
\node[circ] (c) at (6, 0)  [label=below:{ $4$}] {};
\node[circ] (d) at (9, 0)  [label=below:{ $3$}] {};
\node[circ] (e) at (12, 0) [label=below right:{ $1$}] {}
	edge[thick] (a);
\node[circ] (f) at (10.33, 2.5) [label=above right:{ $3$}] {};	
\node[circ] (g) at (9.03, 4.5) [label=above right:{ $4$}] {}
            edge[thick](a);
\node[circ] (t) at (6, 9) [label=above:{ $1$}] {}
            edge[thick] (e)
            edge[thick] (c);
\node[circ] (h) at (1.63, 2.5) [label=above left:{ $2$}] {};
\node[circ] (i) at (3.03,4.5) [label=above left:{ $4$}] {}
             edge[thick] (e);	
\node[circ] (j) at (4.3, 6.5) [label=above left:{ $3$}] {};
\node[circ] (k) at (7.7, 6.5) [label=above right:{ $2$}] {};
\node[circ] (l) at (4.3, 6.5) [label=above left:{ $3$}] {};
\node[circ] (cc) at (6, 3) [label=above right:{ $17$}] {};

\node[circ] (m) at (5, 6) [label=above right:{ $7$}] {}
             edge[thick] (t)
             edge[thick] (j)
             edge[thick] (i)
             edge[thick] (cc);
\node[circ] (n) at (6, 6) [label=above left:{ $6$}] {};
\node[circ] (r) at (7, 6) [label=above left:{ $5$}] {}
             edge[thick] (m)
             edge[thick] (k)
             edge[thick] (g)
             edge[thick] (t)
             edge[thick] (cc);
             
\node[circ] (p) at (3.5, 1) [label=above left:{ $11$}] {}
             edge[thick] (b)
             edge[thick] (c)
             edge[thick] (a)
             edge[thick] (cc);                          
\node[circ] (q) at (6, 1) [label=above left:{ $12$}] {};
\node[circ] (u) at (8.5, 1) [label=above right:{ $13$}] {}
              edge[thick] (p)
              edge[thick] (e)
              edge[thick] (d)
              edge[thick] (c)
              edge[thick] (cc);

\node[circ] (v) at (2, 1) [label= left:{ $10$}] {}
             edge[thick] (p);
\node[circ] (w) at (10, 1) [label=right:{ $14$}] {}
             edge[thick] (u);

\node[circ] (s) at (2.5, 2) [label=below right:{ $9$}] {}
             edge[thick] (v)
             edge[thick] (a)
             edge[thick] (h)
             edge[thick] (i)
             edge[thick] (cc);
\node[circ] (o) at (9.5, 2) [label=below left:{ $15$}] {}
             edge[thick] (w)
             edge[thick] (g)
             edge[thick] (f)
             edge[thick] (e)
             edge[thick] (cc);             

\node[circ] (y) at (4, 4) [label=below:{ $8$}] {}
            edge[thick] (m)
             edge[thick] (s);
\node[circ] (z) at (8, 4) [label=below :{ $16$}] {}
             edge[thick] (o)
             edge[thick] (r);

\begin{pgfonlayer}{background}
\fill[draw, line width=0.3mm,top color=gray!30,bottom color=gray!27] (a.center)
       -- (e.center) -- (t.center)  -- (a.center);
\end{pgfonlayer}
\end{tikzpicture}
\caption{A flag triangulation $\Ind(D)$ of the dunce hat.}
\label{fig:fig1}
\end{figure}
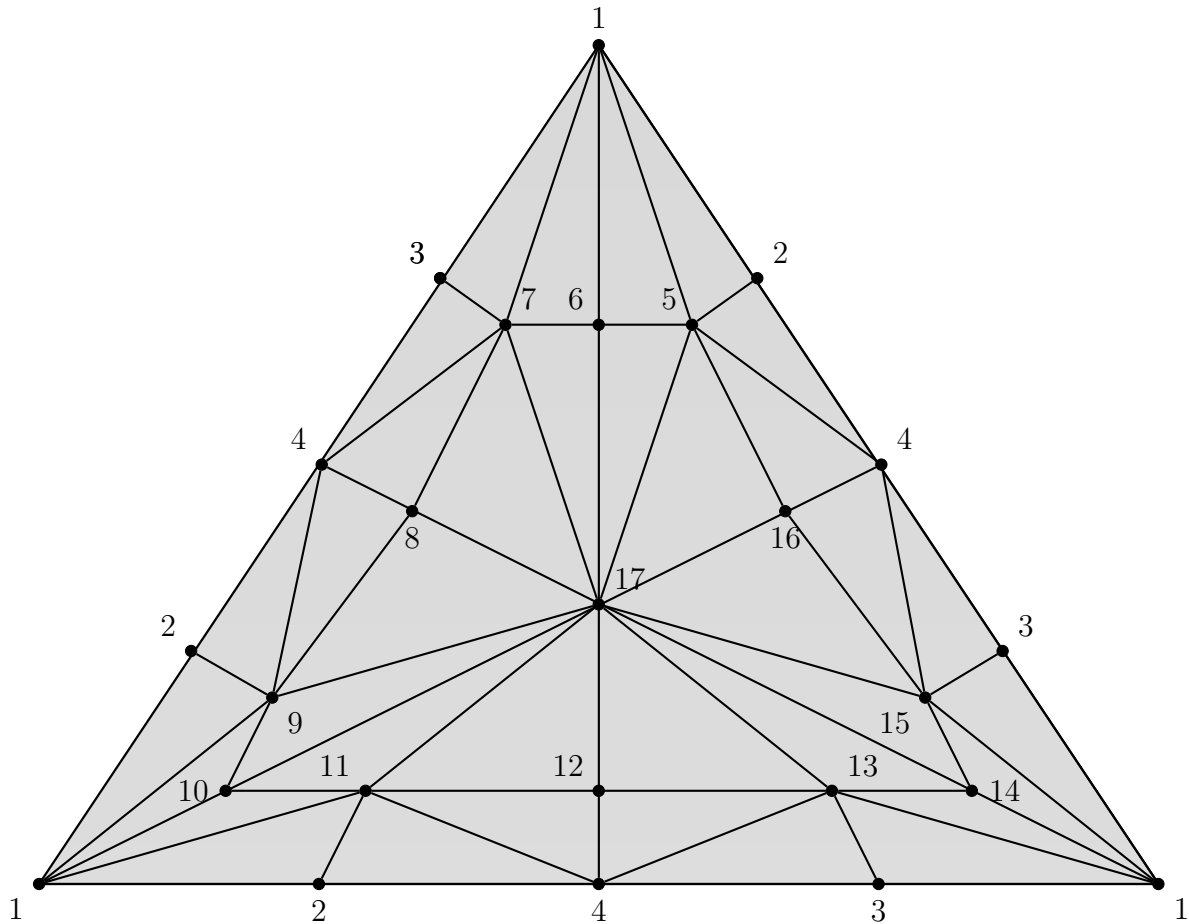

Observe first that $\reg(D)=2$ and $v(D)=3$. The latter follows from the fact that neither any vertex nor any edge is a free face of $\Ind(D)$. For the former, we note that $\overline{C}_{12}$ on the vertex subset $\{5,6,7,8,9,10,11,12,13,14,15,16\}$ is an induced subgraph of $D$ so that
$2=\reg(\overline{C}_{12})\leq \reg(D)$. On the other hand, we have $\reg(D)\leq \cd(D)=2$. Indeed, we set $V_1:=\{5,7,9,11,13,15\}$ and $V_2=\{6,8,10,12,14,16\}$, and define two graphs $R_i$ on $V_1\cup V_2$ by
$E(R_i):=\{pq\in E(D)\colon p\in V_i\;\textnormal{and}\;q\in V_1\cup V_2\}$ for $i\in \{1,2\}$. We then let $Q_1$ and $Q_2$ be the subgraphs of $D$ on $V(D)$ with 
\begin{align*}
E(Q_1):&=E(R_1)\cup \{uv\in E(D)\colon u\in \{2,3\}\},\\
E(Q_2):&=E(R_2)\cup \{xy\in E(D)\colon x\in \{1,4\}\}.
\end{align*}
It is rather easy to check that both subgraphs $Q_1$ and $Q_2$ are co-chordal and satisfy that $E(D)=E(Q_1)\cup E(Q_2)$. 

Now, we construct a graph $G_n$ for each $n\geq 2$ as follows.
The vertex set of $G_n$ is given by $V(G_n)=A_n\cup \bigcup_{i=1}^n B_i$,
where $A_n=\{a_1,\ldots,a_n\}$ with $G_n[A_n]\cong K_n$, and
$B_i=\{y^i_1,\ldots,y^i_{17}\}$ such that $D_i:=G_n[B_i]\cong D$ via the mapping $y^i_j\mapsto j$ for $1\leq i\leq n$ and $1\leq j\leq 17$. Furthermore, the edge set of $G_n$ is given by
\begin{equation*}
E(G_n):=E(K_n)\cup \bigcup_{i=1}^n E(D_i)\cup \{a_iy^i_1\colon 1\leq i\leq n\}.
\end{equation*}
  
We remark that as the graph $D_i$ is connected for each $1\leq i\leq n$, so is the graph $G_n$.  
   
\begin{proof}[{\bf Proof of Theorem~\ref{thm:main}}]
We initially verify that $v(G_n)=3n$ and $\reg(G_n)=2n+1$ for every $n\geq 2$.\medskip

{\em Claim} $1$: $v(G_n)=3n$ for every $n\geq 2$.\medskip

{\em Proof of the Claim} $1$: Firstly, the set $S_n:=\bigcup_{i=1}^n \{y^i_1,y^i_2,y^i_5\}$ is an independent $\ve$-dominating set in $G_n$. Since the set $\{y^i_1,y^i_2,y^i_5\}$ forms a triangle in $\Ind(D_i)$ which is a free face of it, the set $S_n$ is a free face of $\Ind(G_n)$ from which we conclude that $v(G_n)\leq 3n$.

Suppose next that $v(G_n)=i_{\ve}(G_n)=|S|$ for some subset $S\subseteq V(G_n)$. Since $S$ is an independent set and $A_n$ induces a complete subgraph, we have that $|S\cap A_n|\leq 1$. If $S\cap A_n=\emptyset$, it then follows that $|S|\geq 3n$ as $v(D_i)=i_{\ve}(D_i)=3$ for each $i\in [n]$. We may therefore assume that $S\cap A_n=\{a_1\}$ without loss of generality. Since the vertex $a_1$ can not $\ve$-dominate any edge in the induced subgraph $D_i$,
we conclude that $|S\cap B_i|=3$ for each $2\leq i\leq n$. On the other hand, if we consider the graph $L_1:=G_n[B_1\cup \{a_1\}]$, we conclude that $v(L_1)=i_{\ve}(L_1)=3$. This readily follows from the fact that neither any vertex nor any edge in $\Ind(L_1)$ is a free face of it. However, this shows that $|S\cap V(L_1)|=3$; hence, $|S|=3n$.\medskip
 
{\em Claim} $2$: $\reg(G_n)=2n+1$ for every $n\geq 2$.\medskip

{\em Proof of the Claim} $2$: We first show that $\reg(G_n)\leq 2n+1$ by applying to Lemma~\ref{lem:induction-sc} together with an induction on $n$.  

For $n=2$, if we set $L_2:=G_2[B_2\cup \{a_2\}]$, we have that
$\reg(G_2-a_1)=\reg(D_1)+\reg(L_2)$. However, since $\deg_{L_2}(a_2)=1$,
it follows from~\citep[Lemma $6.2$]{BC-jca} that we have either $\reg(L_2)=\reg(L_2-a_2)=\reg(B_2)=2$ or else $\reg(L_2)=\reg(L_2-N_{L_2}[y_1^2])+1$. As a result, we conclude that $\reg(L_2)\leq 3$, which in turn implies the upper bound $\reg(G_2-a_1)\leq 2+3=5$. On the other hand,
there is the isomorphism $G_2-N_{G_2}[a_1]\cong (D_1-y^1_1)\cup D_2$ so that
$\reg(G_2-N_{G_2}[a_1])\leq 2+2=4$. Altogether, these imply that
$\reg(G_2)\leq 5$, which completes the case $n=2$.

For every $n\geq 3$, notice the following isomorphisms
\begin{equation}\label{eq:eq1}
\begin{aligned}
&G_n-a_1\cong G_{n-1}\cup D_1, \\
&G_n-N_{G_n}[a_1]\cong (D_1-y^1_1) \cup D_2\cup \ldots\cup D_n. 
\end{aligned}
\end{equation}

Now, it follows from (\ref{eq:eq1}) together with the induction that
\begin{align*}
&\reg(G_n-a_1)=\reg(G_{n-1})+2\leq 2(n-1)+1+2=2n+1,\\
&\reg(G_n-N_{G_n}[a_1])=2n.
\end{align*}
Thus, we conclude that $\reg(G_n)\leq 2n+1$ for each $n\geq 2$ by Lemma~\ref{lem:induction-sc}. On the other hand, the set
$$M_n:=\{a_1a_2\}\cup \{y^i_4y^i_{17},y^i_7y^i_9\colon 1\leq i\leq n\}$$
forms an induced matching in $G_n$ of size $2n+1$. Therefore, it follows that
$2n+1\leq \im(G_n)\leq \reg(G_n)\leq 2n+1$; hence, $\reg(G_n)=2n+1$ as claimed.

Finally, in order to complete the proof, we set $H_k:=G_{k+1}$ for each $k\geq 1$. It then follows that $v(H_k)=3k+3$ and $\reg(H_k)=2k+3$; thus, $v(H_k)=\reg(H_k)+k$.
\end{proof}
\section{Further comments}
We recall that a simplicial complex $X$ is $k$-collapsible if it can be reduced to the void complex by repeatedly removing a free face of size at most $k$. The collapsibility number $\col(X)$ of $X$ is the smallest integer $k$ such that it is $k$-collapsible. The family of $k$-collapsible simplicial complexes were introduced by Wegner~\cite{Weg}, where he also proved that $\reg(X)\leq \col(X)$ holds for every simplicial complex $X$.
\begin{figure}[ht]
\begin{center}
\begin{tikzpicture}[scale=0.5]

\node [noddee] at (0,0) (v1) [] {};
\node [noddee] at (4,0) (v2) [] {}
	edge [] (v1);
\node [noddee] at (2,2) (v3) [] {}
     edge [] (v2);

\node [noddee] at (0,4) (v4) [] {}     
        edge [] (v1)
        edge [] (v3);
\node [noddee] at (2,6) (v5) [label=above :{ $z$}] {}        
      edge [] (v3)
      edge [] (v4);
\node [noddee] at (4,4) (v6) [] {}         
      edge [] (v2)
      edge [] (v5);

\end{tikzpicture}
\end{center}
\caption{A vertex decomposable graph $R$ with $v(R)=1<2=\col(\Ind(R))$.}
\label{fig:fig2}
\end{figure}
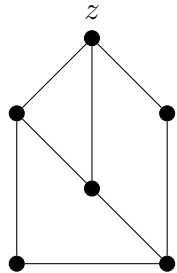

As a result of Corollary~\ref{cor:v-free}, the inequality $v(X)\leq \col(X)$ holds for every simplicial complex $X$. However, it could be strict even for vertex decomposable simplicial complexes. Recall that a simplicial complex $X$ is said to be \emph{vertex decomposable} if it is either a simplex or else there exists a vertex $z$ such that $\del(X;z)$ and $\lk(X;z)$ are vertex decomposable, and every facet of $\del(X;z)$ is a facet of $X$. In the latter, the vertex $z$ is called a \emph{shedding vertex} of $X$. Now, for the graph $R$ depicted in Figure~\ref{fig:fig2}, its independence complex is vertex decomposable, while $v(R)=1<2=\col(\Ind(R))$. 

Finally, we point out that the graph $H_k$ constructed in the proof of Theorem~\ref{thm:main} also provides the first example of a connected graph satisfying that $\col(\Ind(H_k))\geq \reg(H_k)+k$ for every $k\geq 1$ (compare to~\citep[Theorem $1.1 (b)$]{MT}).


\section*{Acknowledgments}
I would like to thank Rafael Villarreal for his invaluable comments and suggestions during preparation of this manuscript.

\end{document}